\documentclass[10pt,a4paper]{article}
\usepackage{amssymb,amsmath,amsopn,amsthm,graphicx,makeidx}
\usepackage[all,2cell]{xy} \UseAllTwocells

\begin{document}

\newtheorem{definition}{Definition}[section]
\newtheorem{definitions}[definition]{Definitions}
\newtheorem{lemma}[definition]{Lemma}
\newtheorem{prop}[definition]{Proposition}
\newtheorem{theorem}[definition]{Theorem}
\newtheorem{cor}[definition]{Corollary}
\newtheorem{cors}[definition]{Corollaries}
\theoremstyle{remark}
\newtheorem{remark}[definition]{Remark}
\theoremstyle{remark}
\newtheorem{remarks}[definition]{Remarks}
\theoremstyle{remark}
\newtheorem{notation}[definition]{Notation}
\theoremstyle{remark}
\newtheorem{example}[definition]{Example}
\theoremstyle{remark}
\newtheorem{examples}[definition]{Examples}
\theoremstyle{remark}
\newtheorem{dgram}[definition]{Diagram}
\theoremstyle{remark}
\newtheorem{fact}[definition]{Fact}
\theoremstyle{remark}
\newtheorem{illust}[definition]{Illustration}
\theoremstyle{remark}
\newtheorem{rmk}[definition]{Remark}
\theoremstyle{definition}
\newtheorem{question}[definition]{Question}
\theoremstyle{definition}
\newtheorem{conj}[definition]{Conjecture}

\newcommand{\stac}[2]{\genfrac{}{}{0pt}{}{#1}{#2}}
\newcommand{\stacc}[3]{\stac{\stac{\stac{}{#1}}{#2}}{\stac{}{#3}}}
\newcommand{\staccc}[4]{\stac{\stac{#1}{#2}}{\stac{#3}{#4}}}
\newcommand{\stacccc}[5]{\stac{\stacc{#1}{#2}{#3}}{\stac{#4}{#5}}}

\renewcommand{\marginpar}[2][]{}

\renewenvironment{proof}{\noindent {\bf{Proof.}}}{\hspace*{3mm}{$\Box$}{\vspace{9pt}}}

\title{Spectra of small abelian categories}
\author{Mike Prest,\\ Alan Turing Building
\\School of Mathematics\\University of Manchester\\
Manchester M13 9PL\\UK\\mprest@manchester.ac.uk}

\maketitle

\tableofcontents

\section{Introduction}\label{secintro}\marginpar{secintro}

Here we take the view that small abelian categories are certain categorical versions of rings and the collection of Serre subcategories of a small abelian category is then analogous to the set of ideals of a commutative ring.  We consider two topologies on the collection of Serre subcategories; one is a very direct lift of the Zariski topology, the other, dual, topology appeared in the context of commutative rings in the work of Thomason \cite{Thom} but, more generally, had already arisen, see \cite{GarkPre3} for the connection, in the work of Ziegler \cite{Zie} on the model theory of modules.  We investigate the primes of these topologies/locales and the corresponding notion of ``local'' abelian category.  We compare with the usual Zariski spectrum on a commutative noetherian ring (which is essentially a special case) and also briefly describe and compare a number of topologies which have been put on the set of isomorphism types of indecomposable injective modules over a commutative ring.

Spectra based on localisation/torsion theories have a long history, see for instance, \cite{GolTops}, \cite{vOVe} and an overview at \cite{GeomNLab}, and there are many ``noncommutative algebraic geometries".  In particular, what we discuss here is related to the work of Rosenberg (e.g.~\cite{Rosen}, \cite{Rosen1}) but, in general, our spectra will have fewer points.  Perhaps the main difference from other work on this topic is that the connection with a larger picture (see \cite{PreRajShv}, \cite{PreADC}) leads us to look only at the torsion theories of finite type in a locally coherent abelian category, equivalently Serre subcategories of small abelian categories.

If ${\mathcal A}$ is a skeletally small abelian category then by $ {\rm Ser}({\mathcal A}) $ we denote the set of Serre subcategories of $ {\mathcal A}$, ordered by inclusion and by $ \langle {\mathcal X}\rangle  $ we denote the Serre subcategory of $ {\mathcal A} $ generated by a collection, ${\mathcal X}$, of objects of ${\mathcal A} $.  Clearly ${\rm Ser}({\mathcal A})$ is a lattice, the points of which give the localisations (up to equivalence) of ${\mathcal A}$, with meet and join given by $ {\mathcal S}\wedge {\mathcal S}'={\mathcal S}\cap {\mathcal S}' $ and $ {\mathcal S}\vee {\mathcal S}'=\langle {\mathcal S}\cup {\mathcal S}'\rangle $.

The two locale structures on ${\rm Ser}({\mathcal A})$ that we will be considering come from the Ziegler and (rep-)Zariski topologies associated to the definable additive category, ${\mathcal D}={\rm Ex}({\mathcal A}, {\bf Ab})$, of exact functors from ${\mathcal A}$ to the category ${\bf Ab}$ of abelian groups (see, for instance, \cite{PreNBK}).  Recall also (\cite[12.10]{PreDefAddCat}) that ${\mathcal A}$ is equivalent to ${\rm fun}({\mathcal D}) = ({\mathcal D},{\bf Ab})^{\rightarrow \prod}$  the category of functors from ${\mathcal D}$ to ${\bf Ab}$ which commute with direct limits and direct products.  We also denote by ${\rm Fun}({\mathcal D}) $ the Ind-completion of ${\rm fun}({\mathcal D}) $ - the, unique-to-equivalence locally coherent Grothendieck category with category of finitely presented objects equivalent to ${\mathcal A}$.

The language of locales is particularly applicable since some of the (induced) maps that we see here are defined only ``at the level of topology'' and not necessarily at the level of points.  Recall that the lattice, ${\rm Op}(T)$, of open subsets of a topological space $T$ forms a {\bf complete Heyting algebra}  ({\bf cHa} for short) (the term ``frame" is also used) - a distributive lattice with infinite joins (hence also infinite meets - for a topology, the interior of the intersection) with meet, $\wedge$, distributing over arbitrary join, $\bigvee$.  With a natural algebraic notion of morphism (namely, a map which preserves finite meets and arbitrary joins) one obtains the category of such algebras.  A {\bf locale} is a complete Heyting algebra but regarded as an object of the opposite category - that is, with arrows going in the direction of continuous maps between topological spaces (noting that a continuous map induces a map in the other direction between the lattices of open sets).  The localic framework (see, e.g., \cite[Chpt.~2]{Joh}) emphasises the view of a topology as its lattice of open sets and it will guide us in identifying the ``primes'' of these locales.

\section{The lattice of Serre subcategories}

The first lemma describes the Serre subcategory generated by a collection of objects, in particular it describes the join operation in ${\rm Ser}({\mathcal A})$.  When, as below, we use the term ``composition chain" we imply nothing (in particular not irreducibility) of the factors.

\begin{lemma}\label{compser}\marginpar{compser} Let ${\mathcal X}$ be any collection of objects in the abelian category ${\mathcal A}$.  Then $A\in \langle {\mathcal X} \rangle$ iff $ A $ has a finite composition chain $ A=A_n\geq A_{n-1}\geq \dots\geq A_0=0 $ with factors $A_{i+1}/A_i$ each of which is a subquotient of an object of $ {\mathcal X}$.

In particular if ${\mathcal S}, {\mathcal S}' \in {\rm Ser}({\mathcal A})$ then ${\mathcal S} \vee {\mathcal S}'$ consists of those objects with a finite composition chain, each factor of which lies in $ {\mathcal S}\cup {\mathcal S}'$.
\end{lemma}
\begin{proof} Certainly $ \langle {\mathcal X}\rangle  $ must contain any such object and the collection of such objects is easily checked to form a Serre subcategory.
\end{proof}

The lattice ${\rm Ser}({\mathcal A})$ is modular, indeed distributive: as in any lattice $ {\mathcal S}\wedge ({\mathcal S}'\vee {\mathcal S}'') \geq  ({\mathcal S}\wedge {\mathcal S}')\vee ({\mathcal S}\wedge {\mathcal S}'') $ and, for the converse, if $  A\in {\mathcal S}\wedge ({\mathcal S}'\vee {\mathcal S}'') $ then, by the lemma, there is $ A=A_n\geq A_{n-1}\geq \dots \geq A_0=0 $ with each $ A_{i+1}/A_i $ in $ {\mathcal S}' $ or $ {\mathcal S}''$; so each of these factors is in $ {\mathcal S}\wedge {\mathcal S}' $ or $ {\mathcal S}\wedge {\mathcal S}''$, as required.

\begin{lemma}\label{SercHa}\marginpar{SercHa} $ {\rm Ser}({\mathcal A}) $ is a complete Heyting algebra (equivalently, a locale).
\end{lemma}
\begin{proof}  One can easily see directly that $ {\rm Ser}({\mathcal A}) $ is a complete lattice with $ \bigwedge _\lambda {\mathcal S}_\lambda =\bigcap _\lambda {\mathcal S}_\lambda  $ and $ \bigvee _\lambda {\mathcal S}_\lambda  =\langle \bigcup _\lambda {\mathcal S}_\lambda \rangle $, to which the obvious modification of the above lemma applies and hence for which the obvious modification of the above argument shows the infinite distributive law $ {\mathcal S}\wedge \bigvee _\lambda {\mathcal S}_\lambda =\bigvee _\lambda {\mathcal S}\wedge {\mathcal S}_\lambda $.
\end{proof}

\begin{example}\label{SerZ}\marginpar{SerZ} Take $ {\mathcal A}={\rm mod}\mbox{-}{\mathbb Z}$. The Serre subcategories are $ 0$ (meaning the full subcategory on all the zero objects), the $ {\mathcal S}_p=\langle {\mathbb Z}_p\rangle  $ for $ p $ a non-zero prime, and arbitrary joins of these, and $ {\mathcal A}=\langle {\mathbb Z}\rangle $. There is a unique maximal proper Serre subcategory, namely that consisting of all torsion groups (the join of all the $ {\mathcal S}_p$) and the interval between that and $ 0 $ is order-isomorphic to the power set of $ {\mathbb N} $ with the $ {\mathcal S}_p $ minimal above $ 0$.
\end{example}

We will need the following observations.  By ${\mathbb A}{\mathbb B}{\mathbb E}{\mathbb X}$ we denote the 2-category of skeletally small abelian categories, exact functors and natural transformations.

\begin{lemma}\label{inser}\marginpar{inser} If $ A,C\in {\mathcal A}\in {\mathbb A}{\mathbb B}{\mathbb E}{\mathbb X} $ then $ C\in \langle A\rangle  $ iff $ C $ has a finite composition chain consisting of subquotients of $ A$.
\end{lemma}
\begin{proof} Clearly any such object must be in $ \langle A\rangle  $ but it is easy to check that the collection of such objects is closed under subobjects, quotients and extensions.
\end{proof}

\begin{lemma}\label{intersectser}\marginpar{intersectser} If $ A,B\in {\mathcal A}\in {\mathbb A}{\mathbb B}{\mathbb E}{\mathbb X} $ are such that $ \langle A\rangle \cap \langle B\rangle \neq 0 $ then $ A $ and $ B $ have a common non-zero subquotient.
\end{lemma}
\begin{proof} Choose a non-zero object $ C\in \langle A\rangle \cap \langle B\rangle $. By the lemma above, $ C $ has a composition chain with factors being subquotients of $ A$; choose a non-zero such factor, $ C'$. Then $ C' $ has a composition chain with factors which are subquotients of $ B$; each such factor is a subquotient both of $ A $ and of $ B$, as required.
\end{proof}

Let $ {\mathcal D}={\rm Ex}({\mathcal A},{\bf Ab}) $ be the definable category corresponding to $ {\mathcal A}\in {\mathbb A}{\mathbb B}{\mathbb E}{\mathbb X}$.  By a definable category we mean a category (equivalent to one) of this form, equivalently a full subcategory of a functor/module category ${\rm Mod}\mbox{-}{\mathcal R}$, where ${\mathcal R}$ is a small preadditive category, which is closed under direct products, direct limits and pure subobjects. To $ {\mathcal D} $ we may associate the complete Heyting algebra, $ {\rm Op}({\rm Zg}({\mathcal D}))$, of open subsets of the Ziegler spectrum, $ {\rm Zg}({\mathcal D})$, of $ {\mathcal D}$, ordered by inclusion.  The {\bf Ziegler spectrum}, ${\rm Zg}({\mathcal D})$ of ${\mathcal D}$ has for its underlying set the set, ${\rm pinj}({\mathcal D})$ of (isomorphism types of) indecomposable pure-injective=algebraically compact objects of ${\mathcal D}$ and has, for a basis of open sets, the sets of the form $(A)=\{ N\in {\rm pinj}({\mathcal D}): A.N\neq 0\}$.  (Since each of ${\mathcal A}$ and ${\mathcal D}$ may be represented as a category of functors on the other, we freely use notation such as $AN=0$.)  The opposite lattice, $ {\rm Cl}({\rm Zg}({\mathcal D}))$, of closed subsets of $ {\rm Zg}({\mathcal D}) $ ordered by inclusion is (e.g.~see \cite[5.1.4, 12.4.1]{PreNBK}) naturally isomorphic to the lattice, $ {\rm Sub}({\mathcal D}) $, of definable subcategories of $ {\mathcal D} $ (again, ordered by inclusion) and so we have the natural isomorphisms:

$${\rm Ser}({\mathcal A}) \simeq {\rm Op}({\rm Zg}({\mathcal D})) \simeq ({\rm Sub}({\mathcal D}))^{\rm op},$$

\noindent with the direct connection from left to right being given by annihilation.  The first isomorphism is given by taking $ {\mathcal S} $ to $ \bigcup _{A\in {\mathcal S}}(A) =  \{ N\in {\rm pinj}({\mathcal D}): AN\neq 0 \mbox{ for some }A\in {\mathcal S}\}$.

A {\bf point} of a locale $ L $ is a Heyting algebra morphism to the two-element locale, hence a prime ideal of $ L$, equivalently its complement - a prime filter of $ L$. Since the lattice $L$ is complete, each prime filter/ideal is principal, so the points of $ L $ correspond to the {\bf prime elements} of $ L$, namely those $ a\in L $ such that $ \{b\in L: b\leq a\} $ is a {\bf prime ideal} (that is, a subset of $L$ which is downwards closed, closed under $\vee$ and such that if $a\wedge b$ belongs to it then either $a$ or $b$ does).

\begin{lemma}\label{primeirred}\marginpar{primeirred} If $ L $ is a locale/cHa then $ a\in L $ is prime iff $ a $ is $\wedge $-{\bf irreducible} (that is, $ a=b\wedge c $ implies $ a=b $ or $ a=c$) iff $ a\geq b\wedge c $ implies $ a\geq b$ or $ a\geq c$.
\end{lemma}
\begin{proof} The first and third conditions are equivalent just by definition and the third implies the second so suppose that the second condition holds and that $ a\geq b\wedge c$. Set $ b'=a\vee b$, $ c'=a\vee c$, so $ b'\wedge c'=(a\vee b)\wedge (a\vee c)= a\vee (b\wedge c) $ (by distributivity) $ =a$, from which we see that $ a=b' $ or $ c' $ and hence $ b\leq a $ or $ c\leq a$.
\end{proof}

For $ {\mathcal A}\in {\mathbb A}{\mathbb B}{\mathbb E}{\mathbb X} $ set $ {\rm Sp}({\mathcal A})=\{ {\mathcal S}\in {\rm Ser}({\mathcal A}): {\mathcal S} \mbox{ is prime }\}$ to be the set of primes = $\wedge$-irreducible elements = ``points'' of $ {\rm Ser}({\mathcal A})$. Note that if $ L $ is the locale of open subsets of a topology then $ U\in L $ is prime iff its complement is an {\bf irreducible} closed set (that is, is not the union of finitely many proper closed subsets), so the natural bijections ${\rm Ser}({\mathcal A}) \simeq {\rm Op}({\rm Zg}({\mathcal D})) \simeq ({\rm Sub}({\mathcal D}))^{\rm op}$ seen above restrict to

\vspace{4pt}

${\rm Sp}({\mathcal A}) \simeq \{ \mbox{open subsets of }{\rm Zg}({\mathcal D})\mbox{ with irreducible complement}\} $

$\simeq \{ \mbox{irreducible closed subsets of }{\rm Zg}({\mathcal D})\}^{\rm op} $

$\simeq \{ \mbox{definable subcategories of }{\mathcal D}\mbox{ with irreducible support}\}^{\rm op}$.

\vspace{4pt}

Given a topological space ${\mathcal T}$, let $ \sim _0 $ denote the equivalence relation which identifies topologically indistinguishable points (that is, those which belong to exactly the same open sets).  The {\bf specialisation preorder} on ${\mathcal T}$ is defined by $ x\leq y $ iff $ x\in {\rm cl}(y) $ iff $ {\rm cl}(x)\subseteq {\rm cl}(y) $ where we use $ {\rm cl} $ to denote closure in a topological space.  Clearly this induces the specialisation partial order on ${\mathcal T}/\sim_0$.  A {\bf generic point} of a closed set is a point whose closure is that set.

\begin{cor}\label{speclorder}\marginpar{speclorder} If every irreducible closed subset of $ {\rm Zg}({\mathcal D})$, where $ {\mathcal D}={\rm Ex}({\mathcal A},{\bf Ab})$, has a generic point then $ {\rm Sp}({\mathcal A})\simeq ({\rm Zg}({\mathcal D})/\sim _0)^{\rm op} $ where the latter is given the specialisation order.
\end{cor}

It is an open question whether the hypothesis of \ref{speclorder} is always satisfied; under a countability assumption it is, see \ref{gencsifctble}.

\begin{example}\label{commnoethirred}\marginpar{commnoethirred} This is the ``classical" case, of the prime spectrum of a commutative noetherian ring, as reconceived by Gabriel \cite{Gab}.  Let $ R $ be a commutative noetherian ring and set $ {\mathcal A}={\rm mod}\mbox{-}R\in {\mathbb A}{\mathbb B}{\mathbb E}{\mathbb X}$. Then

\noindent ($\ast$) $ {\mathcal S}\in {\rm Ser}({\rm mod}\mbox{-}R) $ is irreducible iff $ {\mathcal S}={\rm mod}\mbox{-}R\cap {\mathcal T}_{{\rm cog}(E(R/P))} $ for some $ P\in {\rm Spec}(R)$,

\noindent that is, $ {\mathcal S}\in {\rm Sp}({\rm mod}\mbox{-}R) $ iff there is a prime ideal $ P $ of $ R $ such that $ {\mathcal S}=\{ A\in {\rm mod}\mbox{-}R: (A,E(R/P))=0\}$ where $E(-)$ denotes the injective hull of $(-)$.  By ${\rm cog}(-)$ we mean the hereditary torsion theory cogenerated by $(-)$; we will use ${\mathcal T}_\tau$ and ${\mathcal F}_\tau$ to denote, respectively, the torsion and torsionfree classes of a torsion theory (which, in this paper, always will mean a hereditary torsion theory unless specified otherwise).  For torsion theories see \cite{Ste} (or \cite{PreNBK} for this and other general background that we will need here); we also give a little more background after \ref{commnoethprim}.
The above assertion ($\ast$) is proved as follows.

\noindent ($\Leftarrow $) Suppose that $ A_1,A_2\in {\rm mod}\mbox{-}R\setminus {\mathcal S}$, so there are non-zero morphisms $ f_i:A_i\rightarrow E(R/P)$. Since $ E(R/P) $ is uniform, $ A_0={\rm im}(f_1)\cap {\rm im}(f_2) \neq 0 $ and then $ A_0\in \langle A_1\rangle \cap \langle A_2\rangle  $ which is, therefore, non-zero, as required.

\noindent ($\Rightarrow $) The hereditary torsion theory on $ {\rm Mod}\mbox{-}R $ generated by $ {\mathcal S} $ (as torsion modules) is determined by the (necessarily prime) ideals $ P $ which are maximal with respect to not being in the filter $ {\mathcal U}_{{\mathcal S}} $ of dense ideals of $ R $ (see \cite[VII.3.4]{Ste}) because $ {\mathcal S} $ is cogenerated by the set of corresponding injectives $ E(R/P)$. So let $ P, Q $ be primes maximal with respect to not being in $ {\mathcal U}_{{\mathcal S}}$; we must show that $ P=Q$.

If this were not the case then irreducibility of $ {\mathcal S} $ would imply $ \langle {\mathcal S},R/P\rangle \cap \langle {\mathcal S},R/Q\rangle >{\mathcal S} $ in ${\rm Ser}({\rm mod}\mbox{-}R)$ so pick a cyclic module, $ R/I $ say, in the difference. By factoring out the torsion submodule of $ R/I $ we may suppose that $ R/I $ is torsionfree, that is, $ ({\mathcal S}, R/I)=0$. Since $ R/I\in \langle {\mathcal S},R/P\rangle  $ it follows that there is a non-zero morphism $ f:R/P\rightarrow R/I$; set $ a=f(1+P)$. Since $ ({\mathcal S},R/I)=0 $ it follows by maximality of $ P $ that $ {\rm ann}_R(a)=P $ so, since $ I $ annihilates every element of $ R/I$, $ P=I$. Similarly $ Q=I=P$, as required.
\end{example}

\begin{prop}\label{commnoethprim}\marginpar{commnoethprim} If $ R $ is a commutative noetherian ring then there is a natural bijection $$ {\rm Sp}({\rm mod}\mbox{-}R)\simeq {\rm Spec}(R) $$ which is order-reversing (if we order $ {\rm Spec}(R) $ by inclusion; so order-preserving if we use the specialisation order).
\end{prop}
\begin{proof}  In view of \ref{commnoethirred}, it remains only to note that $ {\mathcal T}_{{\rm cog}(E(R/P))}\subseteq {\mathcal T}_{{\rm cog}(E(R/Q))} $ iff $ {\rm cog}(E(R/P))\supseteq {\rm cog}(E(R/Q)) $ iff $ E(R/Q) $ is not $ E(R/P)$-torsionfree iff $ Q\leq P$.
\end{proof}

Here, since it will be used again, is what we used above. If $ R $ is commutative noetherian then every torsion theory (i.e.~hereditary torsion theory) on $ {\rm Mod}\mbox{-}R $ is {\bf of finite type}, that is, determined by the finitely presented torsion modules, so we have a natural bijection between $ {\rm Ser}({\rm mod}\mbox{-}R)$ and the lattice of finite-type torsion theories on ${\rm Mod}\mbox{-}R$ ordered by inclusion of torsion classes.

For $ P\in {\rm Spec}(R) $ we sometimes write $ E_P=E(R/P) $ and we set $ {\mathcal U}_P=\{ I\leq R: I\nleq P\} =\{ I: (R/I,E_P)=0\} $ - a Gabriel filter of ideals which defines the torsion theory cogenerated by $ E_P$. More generally, if $ X\subseteq {\rm Spec}(R) $ then set $ {\mathcal U}_X=\{ I\leq R:\mbox{ for all }P\in X, I\nleq P\} =\{ I: V(I)\cap X=\emptyset \} = \bigcap _{P\in X}{\mathcal U}_P $ - the Gabriel filter which determines the torsion theory cogenerated by $ \bigoplus _{P\in X}E_P$ (we write $V(I)$ for $\{ P\in {\rm Spec}(R): I\leq P\}$ - a typical closed subset in the Zariski topology on ${\rm Spec}(R)$). Note that the ideals which survive localisation at this torsion theory are the $ P\in X $ and those beneath them.

Conversely, to a torsion theory $ \tau  $ on $ {\rm Mod}\mbox{-}R $ we associate $ D_\tau =\{ P\in {\rm Spec}(R): P\notin {\mathcal U}_\tau \}$; these are the primes $ P $ such that $R/P $ is not $ \tau $-torsion equivalently, since $ R $ is commutative noetherian and so every torsion theory is stable (the injective hull of a torsion module is torsion, see \cite[VII.4.5]{Ste}), such that $ R/P $ is $ \tau $-torsionfree. Note that $ D_\tau =\bigcap _{I\in {\mathcal U}_\tau }D(I) $.

Thus each hereditary torsion theory $ \tau  $ on $ {\rm Mod}\mbox{-}R $, where $R$ is commutative noetherian, is determined by $ {\mathcal U}_\tau \cap {\rm Spec}(R) $ and $ \tau  $ has the form $ \tau _X={\rm cog}(\bigoplus _{P\in X}E_P) $ for some subset $ X $ of $ {\rm Spec}(R)$, which we can make unique by insisting that it is {\bf generalisation-closed} (i.e.~if $ Q\leq P\in X $ then $ Q\in X$). Note that $ {\mathcal T}_X\cap\, {\rm mod}\mbox{-}R=\{ A: (A,E_P)=0\mbox{ for all }P\in X\} = \{ A: A_{(P)}=0\mbox{ for all }P\in X\}$.
Note also that $ R/Q\in {\mathcal F}\in {{\rm cog}(E(R/P))} $ iff $Q\leq P$.

\section{Local abelian categories}\label{seclocabcat}\marginpar{seclocabcat}

We will say that $ {\mathcal A}\in {\mathbb A}{\mathbb B}{\mathbb E}{\mathbb X}$ is {\bf local} if $ 0\in {\rm Sp}({\mathcal A}) $ that is, if the intersection of any two non-zero Serre subcategories of $ {\mathcal A} $ is non-zero. The term `colocal' might at first sight seem more appropriate but, as we will see, abelian categories which are local in our sense have at most one simple object.  Furthermore the example, \ref{commnoethirred}, of commutative noetherian rings shows that factoring out by an irreducible Serre subcategory is analogous to, indeed generalises, localising at (as opposed factoring out by) a prime ideal.  There are three mutually exclusive possibilities for a local abelian category $ {\mathcal A}$; there are plenty examples of the first two cases but it is an open question whether the third can occur.

\noindent case 1: $ {\mathcal A} $ has a simple object.

\noindent case 2: $ {\mathcal A} $ has no simple object but has a minimal non-zero element of $ {\rm Ser}({\mathcal A})$.

\noindent case 3: $ {\mathcal A} $ has no simple object and the intersection of the non-zero Serre subcategories of $ {\mathcal A} $ is $ 0$.

\noindent Let us consider each case in turn.

\vspace{4pt}

\noindent {\bf case 1}: If $ {\mathcal A} $ is local and has a simple object then it has just one, $ S $ say. For $ \langle S\rangle  $ consists of objects having a finite composition series with factors isomorphic to $ S $ so, if $ S' $ were also simple and not isomorphic to $ S $, then we would have $ \langle S\rangle \cap \langle S'\rangle =0$, contradicting irreducibility. Since the interval in $ {\rm Ser}({\mathcal A}) $ between $ \langle S\rangle  $ and $ 0 $ clearly has no other points, it follows that $ \langle S\rangle  $ is the intersection of all the non-zero Serre subcategories of $ {\mathcal A}$.

Let $ {\mathcal D}={\rm Ex}({\mathcal A},{\bf Ab})$, so $ {\mathcal A} $ is $ {\rm fun}({\mathcal D})$. Since $ S $ is a simple, finitely presented object of $ {\rm Fun}({\mathcal D}) $ it is given by a {\bf minimal pair} of pp formulas (we refer elsewhere, e.g.~\cite{PreNBK}, for the model-theoretic view of what we are doing and just give a quick indication here of what is meant): this means that there is a pair $ \psi <\phi  $ of pp formulas such that, for each $ D\in {\mathcal D} $ the interval between $ \psi (D) $ and $ \phi (D) $ in the lattice of pp-definable subgroups of $ D $ either is simple (i.e.~has no more points) or collapses to one point (that is, the pp-pair is closed on $ {\mathcal D}$).  (The formulas can be in any chosen language for the model theory of $ {\mathcal D}$, corresponding to a choice of module category ${\rm Mod}\mbox{-}{\mathcal R}$ containing ${\mathcal D}$ as a definable subcategory, in which case pp formulas are just finitely generated subfunctors of the forgetful functor from ${\rm mod}\mbox{-}{\mathcal R}$ to ${\bf Ab}$.)  It follows by a result of Ziegler (\cite[9.3]{Zie}, see \cite[5.3.6]{PreNBK}) that there is a unique point $ N $ of $ {\rm Zg}({\mathcal D}) $ such that $ SN\neq 0 $; so $ N $ is a topologically isolated point of $ {\rm Zg}({\mathcal D})$, indeed isolated by a minimal pair, to use the relevant terminology. Let $ A $ be any non-zero object of $ {\mathcal A}$. If we had $ AN=0 $ then the annihilator of $ N $ in $ {\mathcal A} $ would be a non-zero Serre subcategory not containing $ S$, a contradiction. We deduce that $ AN\neq 0$, that is, every proper pp-pair is open on $ N $ and hence the definable subcategory of $ {\mathcal D} $ generated by $ N $ is the whole of $ {\mathcal D}$.

\begin{theorem}\label{locabcase1}\marginpar{locabcase1} If $ {\mathcal A}\in {\mathbb A}{\mathbb B}{\mathbb E}{\mathbb X} $ is local and has a simple object $ S $ then, writing $ {\mathcal D}={\rm Ex}({\mathcal A},{\bf Ab}) $ for the corresponding definable category, there is $ N\in {\rm Zg}({\mathcal D}) $ which is isolated, $ \{ N\}=(S)$, and with $ \langle N\rangle ={\mathcal D} $. In particular $ {\rm Zg}({\mathcal D}) $ is irreducible and has $ N $ for its unique generic point. The annihilator of $ S $ in $ {\mathcal D} $ is the unique maximal proper definable subcategory of $ {\mathcal D} $ and contains all other proper definable subcategories of ${\mathcal D}$.
\end{theorem}

The last statement is immediate from the anti-isomorphism (see e.g.~\cite[12.4.1]{PreNBK}) between Serre subcategories of $ {\mathcal A} $ and definable subcategories of $ {\mathcal D}$. In the language of pp formulas it says that, if $ \phi >\psi  $ is, as above, a pp-pair such that $ F_{\phi /\psi }\simeq S$, then the unique maximal proper definable subcategory of $ {\mathcal D} $ is defined by closure of this pp-pair.  By $F_{\phi/\psi}$ is meant the functor defined on objects by $D\mapsto \phi(D)/\psi(D)$.

The simplest example of this case is where $ {\mathcal A} $ is $ k\mbox{-}{\rm mod} $ for some division ring $ k$ and so, as is easily checked, $ {\mathcal D} $ is $ {\rm Mod}\mbox{-}k$.  There are lots of more interesting examples: just take any point $ N $, of the Ziegler spectrum of some definable category, which is isolated in its closure by a minimal pair (if a ring $ R $ has Krull-Gabriel dimension then every point of $ {\rm Zg}_R $ will satisfy this condition, \cite[7.10, 8.4]{Zie}, see \cite[5.3.17, 13.2.1]{PreNBK}) and take $ {\mathcal D} $ to be the definable (sub)category that $ N $ generates and $ {\mathcal A} $ to be $ {\rm fun}({\mathcal D})$.

\vspace{4pt}

\noindent {\bf case 2}: Suppose $ {\mathcal A} $ has a minimal Serre subcategory, $ {\mathcal S}_0 $ say, yet has no simple object. Then $ {\mathcal D}_0={\rm Ex}({\mathcal A}/{\mathcal S}_0,{\bf Ab})={\rm ann}_{{\mathcal D}}{\mathcal S}_0 \subseteq {\mathcal D}={\rm Ex}({\mathcal A},{\bf Ab})$, is the unique maximal proper definable subcategory of $ {\mathcal D} $ and contains all proper definable subcategories. Pick $ N\in {\rm Zg}({\mathcal D})\setminus {\rm Zg}({\mathcal D}_0)$. If $ A\in {\mathcal A} $ is non-zero then $ {\mathcal S}_0\subseteq \langle A\rangle  $ so, as in case 1, $ AN=0 $ would imply $ {\mathcal S}_0\subseteq {\rm ann}_{{\mathcal A}}N $ and hence $ N\in {\mathcal D}_0 $ - contradiction. So every proper pp-pair is open on $ N$ which is therefore {\em a} generic point of $ {\rm Zg}({\mathcal D}) $ (${\rm Zg}_{\mathcal R}$ may contain topologically indistinguishable points though, again, not over a ring with Krull-Gabriel dimension, see e.g.~\cite[5.4.14]{PreNBK}).

\begin{theorem}\label{locabcase2}\marginpar{locabcase2} If $ {\mathcal A}\in {\mathbb A}{\mathbb B}{\mathbb E}{\mathbb X} $ is local, has no simple object, but has a minimal non-zero Serre subcategory $ {\mathcal S}_0 $ then every point of $ {\rm Zg}({\mathcal D})\setminus {\rm ann}_{{\mathcal D}}\,{\mathcal S}_0 $ is a generic point of $ {\rm Zg}({\mathcal D})$; in particular $ {\rm Zg}({\mathcal D}) $ is irreducible and has a generic point.

If $ {\mathcal A} $ skeletally countable, in the sense of having only countably many objects up to isomorphism, then there are at least $ 2^{\aleph_0} $ points in $ {\rm Zg}({\mathcal D})\setminus {\rm ann}_{{\mathcal D}}{\mathcal S}_0$, all of which become identified in the space, $ {\rm Zg}({\mathcal D})/\sim _0$, which results modulo topological equivalence.
\end{theorem}

The last statement is by \cite[8.3]{Zie}, see \cite[5.3.18]{PreNBK}.

There are examples; indeed, there are examples of this case where there are just two (everything and $0$) Serre subcategories/definable subcategories: a ring $ R $ is said to be {\bf indiscrete} if all points of $ {\rm Zg}_R $ are topologically indistinguishable. A simple von Neumann regular ring is indiscrete (and there are non-artinian examples, for instance the endomorphism ring of a countably-infinite-dimensional vector space modulo its socle); non-coherent examples were constructed in \cite{PRZ1}.  For such rings $ {\mathcal A}={\rm fun}\mbox{-}R $ falls under this case 2.

\vspace{4pt}

\noindent {\bf case 3}:  There is no minimal non-zero Serre subcategory of $ {\mathcal A} $, yet $ {\mathcal A} $ is local. It follows that $ \bigcap {\rm Ser}({\mathcal A})=0 $ and so, if $ {\mathcal D}={\rm Ex}({\mathcal A},{\bf Ab}) $ then $ {\rm Zg}({\mathcal D}) $ is the union of its proper closed subsets, so has no generic point, yet it is irreducible. And conversely, any Ziegler-closed set which is irreducible but has no generic point will give an example. The existence of such an example is an open problem. Herzog showed that a countability assumption excludes it; we can rephrase his result as follows.  The original proof is model-theoretic; a functor-category version is given at \cite[5.4.6]{PreNBK}.

\begin{theorem}\label{gencsifctble}\marginpar{gencsifctble} (\cite[4.7]{HerzDual}) Suppose that $ {\mathcal A}\in {\mathbb A}{\mathbb B}{\mathbb E}{\mathbb X} $ is skeletally countable (meaning just countably many arrows; then, if $ {\mathcal A} $ is local, $ {\mathcal A} $ has a minimal non-zero Serre subcategory.
\end{theorem}

We compare briefly with Rosenberg's notion of local abelian category.  That is said in terms of his preodering on objects of an abelian category given by $B\succ A$ iff $A$ is a subquotient of $B^n$ for some $n$ (this implies that $A$ is in the Serre subcategory generated by $B$, but not conversely since the class of such $A$, given $B$, will not in general be closed under extensions).  A non-zero object $A$ of an abelian category ${\mathcal A}$ is {\bf quasifinal} if $B\succ A$ for every non-zero object $B$ of ${\mathcal A}$ and ${\mathcal A}$ is {\bf local} in Rosenberg's sense if it has a quasifinal object.  If $A$ is quasifinal then it is easily seen that $\langle A \rangle$ is the unique minimal non-zero Serre subcategory of ${\mathcal A}$ so, if ${\mathcal A}$ is local in Rosenberg's sense, then it is local in the sense of this paper.

\section{The Zariski locale}\label{seczar}\marginpar{seczar}

Now we turn to the, dual, rep-Zariski topology.  Let $ {\mathcal A}\in {\mathbb A}{\mathbb B}{\mathbb E}{\mathbb X} $ and set $ {\mathcal D}={\rm Ex}({\mathcal A},{\bf Ab})$.

The rep-Zariski topology on $ {\rm pinj}({\mathcal D}) $ has basic open sets the $ [A]=\{ N\in {\rm pinj}({\mathcal D}): AN=0\} $ for $ A\in {\mathcal A} $ and so a typical open set has the form $ \bigcup _\lambda [A_\lambda ] $ since $ [A]\cap [B]=[A\oplus B]$.

Define the {\bf Zariski topology} on $ {\rm Ser}({\mathcal A}) $ by taking for a basis of open sets the $ [A]=\{ {\mathcal S}\in {\rm Ser}({\mathcal A}):A\in {\mathcal S}\} =\{{\mathcal S}\in {\rm Ser}({\mathcal A}):\langle A\rangle \subseteq {\mathcal S}\} $ for $ A\in {\mathcal A} $ (we are recycling the notation but the meaning should be clear from the context). That is, the basic opens are the up-intervals above finitely generated Serre subcategories. (Note that the finitely generated $ {\mathcal S}_0\in {\rm Ser}({\mathcal A}) $ are the finite (also termed compact) elements of $ {\rm Ser}({\mathcal A}) $ in the sense of e.g.~\cite[p.~63]{Joh}.) The typical open set in the Zariski topology on $ {\rm Ser}({\mathcal A}) $ therefore has the form $ \bigcup _\lambda \langle A_\lambda \rangle =\{ {\mathcal S}\in {\rm Ser}({\mathcal A}): A_\lambda \in {\mathcal S}$ for some $ \lambda \}$.

The next result allows us to regard $ {\rm Ser}({\mathcal A}/{\mathcal S}) $ as a subspace of $ {\rm Ser}({\mathcal A}) $.

\begin{lemma}\label{reltoploc}\marginpar{reltoploc} If ${\mathcal S} \in {\rm Ser}({\mathcal A})$ then the rep-Zariski topology on ${\rm Ser}({\mathcal A}/{\mathcal S})$ and the topology induced from ${\rm Ser}({\mathcal A})$ coincide.
\end{lemma}
\begin{proof} Given a basic open $[S]$ on ${\rm Ser}({\mathcal A})$ we have $[S]\,\cap\, {\rm Ser}({\mathcal A}/{\mathcal S}) =[S_{\mathcal A}]$ which is a basic open in the topology on ${\rm Ser}({\mathcal A}/{\mathcal S})$.  Conversely, if $S'\in {\mathcal A}/{\mathcal S}$ then there is $S\in {\mathcal S}$ with $S_{\mathcal A} =S'$ and hence with $[S]\,\cap\, {\rm Ser}({\mathcal A}/{\mathcal S}) = [S']$.
\end{proof}

In the case that $ {\mathcal S}=\langle A_0\rangle $ is finitely generated then $ {\rm Ser}({\mathcal A}/{\mathcal S}) =[A_0]$ under the identification above and hence it, and its open subsets, are actually open in $ {\rm Ser}({\mathcal A})$.

\begin{example}\label{zarserZ}\marginpar{zarserZ} The Zariski topology on $ {\rm Ser}({\rm mod}\mbox{-}{\mathbb Z})$ has, for basic open sets:

\noindent $ [0]={\rm Ser}({\rm mod}\mbox{-}{\mathbb Z})$;

\noindent $[{\mathbb Z}_{p_1}\oplus \dots \oplus {\mathbb Z}_{p_n}] $ for $ p_1,\dots,p_n $ any non-zero primes;

\noindent $ [{\mathbb Z}]$.

Similar is $ {\rm Ser}({\rm mod}\mbox{-}{\mathbb Z}_{(p)}) $ where we have:

\noindent $ [0]={\rm Ser}({\rm mod}\mbox{-}{\mathbb Z}_{(p)})=\{ \langle 0\rangle , \langle {\mathbb Z}_{(p)}\rangle , {\rm mod}\mbox{-}{\mathbb Z}_{(p)}\}$;

\noindent $ [{\mathbb Z}_p]=\{ \langle {\mathbb Z}_{(p)}\rangle , {\rm mod}\mbox{-}{\mathbb Z}_{(p)}\}$;

\noindent $ [{\mathbb Z}_{(p)}]=\{ {\rm mod}\mbox{-}{\mathbb Z}_{(p)}\}$,

\noindent the difference here being that $ \{ {\mathbb Q}\} $ is open.
\end{example}

The above example illustrates the general point that this topology $ ({\rm Ser}({\mathcal A}), {\rm Zar}({\rm Ser}({\mathcal A}))) $ is equivalent to $ {\rm Zar}({\mathcal D})$: $$ {\rm Op}({\rm Ser}({\mathcal A}), {\rm Zar}({\rm Ser}({\mathcal A})))\simeq {\rm Op}({\rm Zar}({\mathcal D})) $$ because the basic opens correspond and since $ \langle A\rangle \subseteq \langle B\rangle $, equivalently $ [A]_{{\rm Ser}({\mathcal A})}\supseteq [B]_{{\rm Ser}({\mathcal A})} $ iff $ {\rm ann}_{{\rm pinj}({\mathcal D})}(A)\supseteq {\rm ann}_{{\rm pinj}({\mathcal D})}(B) $ iff $ [A]_{{\rm Zar}({\mathcal D})}\supseteq [B]_{{\rm Zar}({\mathcal D})}$. Thus the rep-Zariski topology on $ {\rm pinj}({\mathcal D}) $ may also be realised as a topology on $ {\rm Ser}({\mathcal A})$. It can even be seen as a topology on the (set of isomorphism classes of the) objects of $ {\mathcal A}$: if we preorder $ {\mathcal A} $ by setting, for $ A, B\in {\mathcal A}$, $ A\leq B $ iff $ A\in \langle B\rangle  $ then the opens of the Zariski topology on $  {\rm Ser}({\mathcal A}) $ correspond exactly to the upwards-closed subsets in this ordering (the Alexandrov topology, see \cite[p.~45]{Joh}, which is the finest for which the poset order is exactly the specialisation order in the topology).

As we did with the Ziegler locale on $ {\rm Ser}({\mathcal A}) $, we consider the primes of the Zariski locale that is, (\ref{primeirred}), the open sets which are not the intersection of two larger open sets.  Let us say that $ A\in {\mathcal A} $ is {\bf Serre-local} if for all $ {\mathcal S},{\mathcal S}'\in {\rm Ser}({\mathcal A}) $ if $ A\in {\mathcal S}\vee {\mathcal S}' $ then either $ A\in {\mathcal S} $ or $ A\in {\mathcal S}' $ - exactly the condition that the Zariski-open set $ [A] $ be prime in the Zariski locale.

\begin{lemma}\label{serloc0}\marginpar{serloc0} The object $ A\in {\mathcal A} $ is Serre-local iff $ \bigvee \{{\mathcal S}\in {\rm Ser}({\mathcal A}): {\mathcal S}<\langle A\rangle \}$ is a proper subcategory of $\langle A\rangle $.
\end{lemma}
\begin{proof} If $ \langle A\rangle  $ were equal to the join of its proper Serre subcategories then, by \ref{compser}, it would be the join of finitely many of them, so could not be Serre-local. The converse follows from the fact that, by distributivity of $ {\rm Ser}({\mathcal A})$, $ A $ belongs to neither $ \langle B\rangle  $ nor $ \langle C\rangle  $ iff  it does not belong to $ (\langle A\rangle \wedge \langle B\rangle )\vee (\langle A\rangle \wedge \langle C\rangle )$.
\end{proof}

\begin{lemma}\label{serlocobj}\marginpar{serlocobj} The object $ A\in {\mathcal A} $ is Serre-local

\noindent iff for every composition chain $ A=A_n\geq \dots\geq A_0=0 $ there is some $ i $ such that $ \langle A_{i+1}/A_i\rangle =\langle A\rangle  $

\noindent iff for every $ {\mathcal S}\in {\rm Ser}({\mathcal A}) $ and for every such composition chain $ A=A_n\geq \dots\geq A_0=0 $ of $ A$, if $ A\in {\mathcal S} $ then there is $ i $ with $  A_{i+1}/A_i\in {\mathcal S} $

\noindent iff there is a unique maximal ${\mathcal S}\in {\rm Ser}({\mathcal A}) $ with $ A\notin {\mathcal S}$.
\end{lemma}

The first two equivalences are immediate. For the last, we argue as just above.
Note that, given any non-zero $ A\in {\mathcal A} $ there is, by Zorn's Lemma, at least one Serre subcategory which is maximal with respect to not containing it (and any such Serre subcategory will be (Ziegler-)prime, i.e.~in ${\rm Sp}({\mathcal A})$).

If $ A\in {\mathcal A}\in {\mathbb A}{\mathbb B}{\mathbb E}{\mathbb X} $ and $ {\mathcal S}\in {\rm Ser}({\mathcal A}) $ let us say that $ A $ is $ {\mathcal S}$-{\bf simple} if the localisation, $ A_{{\mathcal S}} $, of $A$ at ${\mathcal S}$ (that is, the image of $A$ under the canonical functor ${\mathcal A} \rightarrow {\mathcal A}/{\mathcal S}$) is a simple object of $ {\mathcal A}/{\mathcal S}$.

\begin{lemma}\label{critobj0}\marginpar{critobj0} The object $ A $ is $ {\mathcal S}$-simple iff $ A\notin {\mathcal S} $ and if, for any exact sequence $ 0\rightarrow A'\rightarrow A\rightarrow A''\rightarrow 0 $ in $ {\mathcal A}$, exactly one of $ A', A'' $ is in $  {\mathcal S}$. If this is so then there is no Serre subcategory strictly lying between $ {\mathcal S} $ and $ \langle A,{\mathcal S}\rangle $.
\end{lemma}
\begin{proof} The first statement is clear since localisation is exact. For the second, suppose we had $ {\mathcal S}<{\mathcal S}'<\langle A,{\mathcal S}\rangle $. Then, by \ref{compser}, a non-zero subquotient of $ A $ would be in $ {\mathcal S}'\setminus {\mathcal S}$, say $ A\geq A'>A''\geq 0 $ with $ A'/A''\in {\mathcal S}'\setminus {\mathcal S}. $ Since $ A_{{\mathcal S}} $ is simple it follows that both $ A/A' $ and $ A'' $ are in $ {\mathcal S}$, but then $ A $ would be in $ {\mathcal S}' $ - contradiction.
\end{proof}

Let us say that $ A $ is {\bf quasisimple} if $ A $ is $ {\mathcal S}_{<A}$-simple, where $ {\mathcal S}_{<A}=\bigvee \{ {\mathcal S}\in {\rm Ser}({\mathcal A}): {\mathcal S}<\langle A\rangle \}$; this already implies (by \ref{serloc0}) that $ A $ is Serre-local.  This is a marginally more general notion than that of a cocritical object (see, e.g., \cite{GolChain}) in torsion theory and the comment just made is the observation that our notion of prime in the Zariski locale includes (at least, in the finite-type, coherent context within which we work) that of prime torsion theories in the sense of, e.g., \cite{GolTops}.

\begin{lemma}\label{critobj}\marginpar{critobj} An object $ A\in {\mathcal A} $ is quasisimple

\noindent iff there is some $ {\mathcal S}\in {\rm Ser}({\mathcal A}) $ such that $ A_{{\mathcal S}} $ is simple in $ {\mathcal A}/{\mathcal S}$ and, for every $ {\mathcal S}\in {\rm Ser}({\mathcal A}) $ and for every exact sequence $ 0\rightarrow A'\rightarrow A\rightarrow A''\rightarrow 0$, if $ A\in {\mathcal S} $ then either $ A'\in {\mathcal S} $ or $ A''\in {\mathcal S} $

\noindent iff $ A $ is Serre-simple and there is $ {\mathcal S}\in {\rm Ser}({\mathcal A}) $ such that $ A_{{\mathcal S}} $ is simple in $ {\mathcal A}/{\mathcal S} $ (in which case we may take $ {\mathcal S}={\mathcal S}_{<A}$).
\end{lemma}
\begin{proof} The only observation needed is that if $ A_{{\mathcal S}} $ is simple then $ A_{{\mathcal S}\wedge \langle A\rangle } $ is simple.
\end{proof}

Next, we compare primes for the rep-Zariski and Ziegler topologies.  If $R$ is commutative noetherian then the definable category corresponding to ${\mathcal A} =R\mbox{-}{\rm mod}$ is ${\rm Inj}\mbox{-}R$ and there is a natural bijection between the various sets (usual, rep-Zariski, Ziegler) of primes (\cite[14.4.13]{PreNBK}).  In the general context it is also the case that if $A\in {\mathcal A}$ is such that $[A]$ is (rep-)Zariski-prime, that is, if the Ziegler-open subset $(A)$ of ${\rm Zg}({\mathcal D})$ is $\cup$-irreducible, then it is the rep-Zariski-closure of a point, $N$, say (\cite[14.2.6]{PreNBK}).  More generally if $U$ is an open set in the Zariski locale such that its (Zariski-irreducible) complement has a Zariski-generic point, then that point also is, of course, a generic point of its own Ziegler-closure.  Since two points are indistinguishable in the Ziegler topology iff they are indistinguishable in the rep-Zariski topology, this gives an injection from those primes of the Zariski locale where the corresponding closed set has a Zariski-generic point to primes of the Ziegler locale.  Conversely, if an irreducible closed set of the Ziegler locale has a generic point $N$ then the intersection, $\bigcap_{N\in (A)} (A)$ of its Ziegler-open neighbourhoods is a closed set in the Zariski locale which is (clearly) irreducible.

\begin{prop}\label{primescorresp}\marginpar{primescorresp}  If ${\mathcal D}$ is a definable category then the points of ${\rm pinj}({\mathcal D})$ induce a natural bijection between those primes for the Ziegler and Zariski locales which have associated (to the corresponding $\cup$-irreducible closed sets) generic points.
\end{prop}

This, together with the identification in ``classical" situations of the primes in our sense with primes in the classical sense, makes it reasonable to call the above the primes of the category ${\mathcal A}$ (while allowing that we may be missing some points which exist at the locale but not at the point level and, of course, noting that there are more general notions of primes of an abelian category).

\section{Exact functors}\label{secex}\marginpar{secex}

Any exact functor $ f:{\mathcal A}\rightarrow {\mathcal B} $ between abelian categories induces a map $ {\rm Ser}(f):{\rm Ser}({\mathcal B})\rightarrow {\rm Ser}({\mathcal A}) $ which is defined by taking $ {\mathcal T}\in {\rm Ser}({\mathcal S}) $ to $ f^{-1}{\mathcal T}=\{ A\in {\mathcal A}: fA\in {\mathcal T}\}$. Since $ f $ is exact, $ f^{-1}\in {\rm Ser}({\mathcal A})$. Clearly $  {\rm Ser}(f) $ is order-preserving and $ {\rm Ser}(f) $ commutes with meets - $ f^{-1}({\mathcal S}_1\cap {\mathcal S}_2) =f^{-1}{\mathcal S}_1\cap f^{-1}{\mathcal S}_2 $ - but not with joins.

\begin{example}\label{serjoin}\marginpar{serjoin} Let $ {\mathcal B}={\rm mod}\mbox{-}R $ where $ R $ is a finite-dimensional hereditary algebra with a non-homogeneous tube of rank, say, 2, with quasi-simples $ S_1, S_2 $ and let $ T $ be the extension of $ S_2 $ by $ S_1 $ given by the non-split exact sequence $ 0\rightarrow S_1\rightarrow T\rightarrow S_2\rightarrow 0$. Let $ {\mathcal A} $ be the full subcategory whose objects are $ T $ and its finite self-extensions. Both $ {\mathcal A} $ and $ {\mathcal B} $ are abelian and the inclusion $ {\mathcal A}\rightarrow {\mathcal B} $ is exact. Let $ {\mathcal S}_i $ be the Serre subcategory of $ {\mathcal B} $ generated by $ S_i$; clearly $ f^{-1}({\mathcal S}_1\vee {\mathcal S}_2) ={\mathcal A} $ but $ f^{-1}{\mathcal S}_1=0=f^{-1}{\mathcal S}_2 $ so $ f^{-1}({\mathcal S}_1\vee {\mathcal S}_2)\neq f^{-1}{\mathcal S}_1\vee f^{-1}{\mathcal S}_2$.
\end{example}

\begin{lemma}\label{serfcts}\marginpar{serfcts} Let $ f:{\mathcal A}\rightarrow {\mathcal B} $ be a morphism in $ {\mathbb A}{\mathbb B}{\mathbb E}{\mathbb X}$. Then $ {\rm Ser}(f):{\rm Ser}({\mathcal B})\rightarrow {\rm Ser}({\mathcal A}) $ is a continuous map for the Zariski locale.
\end{lemma}
\begin{proof} Take a basic open $ [A] $ of $ {\rm Zar}({\mathcal A})$; then we have $ ({\rm Ser}(f))^{-1}[A] = \{{\mathcal T}\in {\rm Ser}({\mathcal B}): fA\in {\mathcal T}\} =[fA]$.  Then $ ({\rm Ser}(f))^{-1}\bigcup _\lambda [A_\lambda ] = \{ {\mathcal T}: fA_\lambda \in {\mathcal T}\mbox{ for some }\lambda \} = \bigcup _\lambda [fA_\lambda ]$.

To check continuity with respect to the Ziegler locale, since the $[A]$ are the basic Ziegler-closed sets, it remains to note that $ ({\rm Ser}(f))^{-1}\bigcap _\lambda [A_\lambda ] = \{ {\mathcal T}: fA_\lambda \in {\mathcal T}\mbox{ for all }\lambda \} = \bigcap _\lambda [fA_\lambda ]$.
\end{proof}

\begin{example}\label{serfnotsp}\marginpar{serfnotsp} It need not be that $ {\rm Ser}(f) $ restricts to a map from $ {\rm Sp}({\mathcal B}) $ to $ {\rm Sp}({\mathcal A})$. (This might seem a defect if one has commutative rings in mind but it is already seen in noncommutative rings: consider the embedding of the ring  $\left( \begin{array}{cc} k & k \\ 0 & k \end{array} \right) $  into  $\left( \begin{array}{cc} k & k \\ k & k \end{array} \right) $: this does not induce a map from the prime ideals of the second ring to that of the first.)
For an example, let $ R=k[A_3] $ where the quiver $ A_3 $ is given the orientiation shown

$
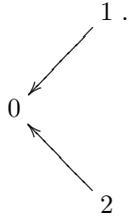
\xymatrix{& 1 \ar[dl] \\ 0 \\ & 2 \ar[ul]}$.

\noindent The Auslander-Reiten quiver of $ R $ is below, labelled in terms of simples, projectives and injectives.

$
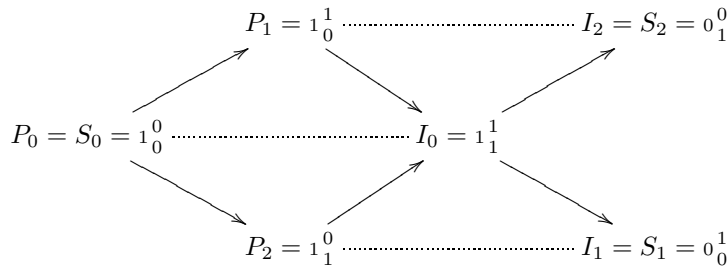
\xymatrix{& P_1={\scriptstyle 1} \stac{1}{0} \ar@{.}[rr] \ar[dr] & & I_2=S_2= {\scriptstyle 0} \stac{0}{1} \\ P_0=S_0= {\scriptstyle 1} \stac{0}{0} \ar[ur] \ar[dr] \ar@{.}[rr] & & I_0={\scriptstyle 1} \stac{1}{1} \ar[ur] \ar[dr] \\ & P_2={\scriptstyle 1} \stac{0}{1} \ar[ur] \ar@{.}[rr] & & I_1=S_1= {\scriptstyle 0} \stac{1}{0}} $

\noindent  Let $ {\mathcal A} $ be the full subcategory on the projective modules of the form $ P_1^n\oplus P_2^m$. Since $ {\rm End}(P_i)=k $ for each $ i $ and $ (P_1,P_2)=0=(P_2,P_1) $ this is an abelian (semisimple) category. Let $ {\mathcal B}'={\rm mod}\mbox{-}R $ and consider the inclusion $ f':{\mathcal A}\rightarrow {\mathcal B}'$. Note that $ 0\notin {\rm Sp}({\mathcal A}) $ since $ \langle P_1\rangle _{{\mathcal A}}\cap \langle P_2\rangle _{{\mathcal A}}=0$; also $ \langle fP_1\rangle _{{\mathcal B}'}\cap \langle fP_2\rangle _{{\mathcal B}'}=\langle P_0\rangle _{{\mathcal B}'}$. This is not yet an example, but we form the quotient category $ {\mathcal B}={\rm mod}\mbox{-}R/\langle S_1,S_2\rangle $; this is non-zero since $ S_0\notin \langle S_1,S_2\rangle $. Let $ f $ be the composition of $ f $ with the quotient $ \pi :{\mathcal B}'\rightarrow {\mathcal B}$. Since $ {\rm Ser}({\mathcal B}) $ is, \ref{reltoploc}, isomorphic to the interval above $ \langle S_1,S_2\rangle  $ in $ {\rm Ser}({\mathcal B}')$, it is a two-point lattice so certainly $ 0\in {\rm Ser}({\mathcal B}) $ is irreducible. But $ {\rm Ser}(f)0=0 $ (since the images of both $ P_1 $ and $ P_2 $ under $ \pi  $ are isomorphic to the image of $ S_0 $ under $ \pi $) which is not in $ {\rm Sp}({\mathcal A})$, as required.
\end{example}

\section{Topologies on $ {\rm inj}_R$}\label{secinj}\marginpar{secinj}

If ${\mathcal C}$ is an additive category then we denote by $ {\rm Inj}({\mathcal C})$ the full subcategory of injective objects of ${\mathcal C}$ and by ${\rm inj}({\mathcal C})$ the class (which in the categories we consider will be a set) of isomorphism types of indecomposable injective objects of ${\mathcal C}$.  In the case that ${\mathcal C} = {\rm Mod}\mbox{-}R$ we write ${\rm Inj}\mbox{-}R$ and ${\rm inj}_R$ respectively.  We recall that the results of Gruson and Jensen \cite{GrJe1} mean that in treating these we are, by proxy, treating the pure-injectives of definable categories in the following sense.

If ${\mathcal D}$ is a definable subcategory of ${\rm Mod}\mbox{-}{\mathcal R}$ then there is a full and faithful embedding ${\rm Mod}\mbox{-}{\mathcal R} \rightarrow ({\mathcal R}\mbox{-}{\rm mod}, {\bf Ab})$ which takes an object $M_{\mathcal R}$ to the tensor functor $_{\mathcal R}L \mapsto M\otimes_{\mathcal R} L$.  This functor takes pure exact sequences to exact sequences and induces an equivalence of categories of ${\rm Mod}\mbox{-}{\mathcal R} $ with the absolutely pure = fp-injective objects of $ ({\mathcal R}\mbox{-}{\rm mod}, {\bf Ab})$.  This equivalence restricts to ${\rm Pinj}\mbox{-}{\mathcal R} \simeq {\rm Inj}\big( ({\mathcal R}\mbox{-}{\rm mod}, {\bf Ab})\big)$ and hence induces a natural bijection ${\rm pinj}({\mathcal D}) \leftrightarrow {\rm inj}\big(({\mathcal R}\mbox{-}{\rm mod}, {\bf Ab})\big)$.  The image of ${\mathcal D}$ under the above embedding forms the class of absolutely pure objects which are torsionfree for a certain torsion theory, $\tau^{\rm d}_{\mathcal D}$, of finite type on the functor category.  If we then compose the embedding of ${\rm Mod}\mbox{-}{\mathcal R}$ into the functor category with localisation at this torsion theory then we obtain an equivalence of ${\mathcal D}$ with the category of absolutely pure objects of the (locally coherent) quotient category ${\mathcal G} =({\mathcal R}\mbox{-}{\rm mod}, {\bf Ab})_{\tau^{\rm d}_{\mathcal D}} $ (see, e.g., \cite[\S\S 12.1, 12.3]{PreNBK})).  That equivalence restricts in particular to an equivalence of ${\rm Pinj}({\mathcal D})$ with ${\rm Inj}({\mathcal G})$ and a natural bijection between ${\rm pinj}({\mathcal D})$ and ${\rm inj}({\mathcal G})$.  It was shown in \cite[\S5]{PreRajShv} that if we give each of the latter two sets either the Ziegler topology or the rep-Zariski topology then the bijection is a homeomorphism.  Thus the study of Ziegler and rep-Zariski spectra is actually subsumed under the apparently special case of the Ziegler and rep-Zariski spectra restricted to points in ${\rm inj}({\mathcal G})$ of a locally coherent category ${\mathcal G}$.

We give a brief account of the descriptions and basic properties of various related topologies which may be put on the set of indecomposable injectives of a locally finitely presented Grothendieck category; we will consider just the case that the category is ${\rm Mod}\mbox{-}R$ for some ring but the changes that have to be made for categories ${\rm Mod}\mbox{-}{\mathcal R}$ where ${\mathcal R}$ is a skeletally small preadditive category are minor.  The general case can then be obtained using that every locally finitely presented Grothendieck category is a finite type localisation of a module category ${\rm Mod}\mbox{-}{\mathcal R}$ (see \cite[11.1.27]{PreNBK}).

These topologies fall into two classes: Ziegler- and Zariski-type topologies which are roughly dual in the sense that basic closed sets in the one type of topology become basic open sets in the other. In cases where the prime spectrum, $ {\rm Spec}(R) $ of $R$, may be identified with a subset of $ {\rm inj}_R $ the Zariski-type topologies give exactly, or something like, the usual Zariski topology on $  {\rm Spec}(R)$.

Ziegler-type topologies have basic open sets of the form

$(M)=\{ E\in {\rm inj}_R: (M,E)\neq 0\}$

\noindent for certain choices of $R$-module $ M$, whereas

Zariski-type topologies have basic open sets of the form

$[M]= (M)^{\rm c}=\{ E\in {\rm inj}_R: (M,E)\neq 0\}$.

\subsection{Ziegler-type topologies on $ {\rm inj}_R$}

Since, by injectivity of the codomain, $ (M)=\bigcup \{ (M'): M'\leq M, M' \mbox{ finitely generated}\} =\bigcup \{(R/I): I\leq R_R, (R/I,M)\neq 0\}$, in order to specify a basis of a Ziegler-type topology it is enough to say for which right ideals $ I $ is the set $ (R/I) $ open.

We recall the following description of a basis of the Ziegler topology on $ {\rm inj}_R$; this can be seen as a kind of (infinitary if $R$ is not right coherent) elimination of quantifiers for injectives.  If $\phi$ is a pp formula for right modules then its elementary dual, $D\phi$, is a formula for left modules (see, e.g., \cite[\SS 1.3, 10.3]{PreNBK}); in particular $D\phi(_RR)$ will be a right ideal of $R$.

\begin{theorem}\label{injbasis1}\marginpar{injbasis1} \cite[7.3]{PreRajShv} Suppose that $ \phi ,\psi  $ are pp formulas in one free variable; then $ (\phi /\psi )\cap {\rm inj}_R=\bigcup \{ (\phi _r): r\in D\psi (_R R)\setminus D\phi (_R R)\}  $ where $ \phi _r $ is the pp formula $ \exists y \,(x=yr\wedge \phi (y))$, which is such that $ \phi _r(M)=\phi (M)r$.
\end{theorem}

In the case that $ R $ is right coherent, so $ {\rm inj}_R $ is a closed subset of $ {\rm pinj}_R $ in the Ziegler topology and hence is compact, the union will be a finite one.  We slightly extend some of the other results of \cite[\S 7]{PreRajShv}.

\vspace{4pt}

\noindent 1. The finest Ziegler-type topology has basis $ (M) $ with $ M $ arbitrary, equivalently the $ (R/I) $ with $ I\leq R_R $ any right ideal. We observe that this is the {\bf full support topology} of Burke \cite{BurThes}, \cite{Burpp}. That was defined on the whole of $ {\rm pinj}_R $ to have basis the sets $ (p/\psi ) $ for $ p $ a pp-type and $ \psi  $ a pp formula, where an indecomposable pure-injective $ N $ is in $ (p/\psi ) $ iff $N$ contains an element which satisfies all the pp formulas in $ p $ but does not satisfy $ \psi $.

\begin{lemma}\label{fullsupp}\marginpar{fullsupp} Suppose that $ p=p(x) $ is a pp-type in one free variable and set $ I_p=\sum _{\phi \in p}D\phi (_R R)\leq R_R$.
Then $ (p)=(R/I_p) $ (where $ (p) $ means $ (p(x)/(x=0))$).  More generally, if $ \psi  $ is pp then $ (p/\psi )=\big((D\psi (_R R)+I_p)/I_p\big)$.
\end{lemma}
\begin{proof} If $E$ is indecomposable injective then $ E\in (p) $ iff there is $ a\in E$, $a\neq 0 $ such that $ a\in p(E)$ (the solution set of $p$ in $E$), that is, iff $ a\in \phi (E) $ for all $ a\in p$, equivalently, see \cite[1.1]{PRZ1}, iff $ aD\phi (_R R)=0 $ for all $ \phi \in p$, that is, iff $ aI_p=0$. For the more general case one adds the condition $ a\notin \psi (E)$, equivalently $ aD\psi (_R R)\neq 0 $ and the result follows.
\end{proof}

Burke (\cite[\S7.4]{BurThes}, see \cite[5.3.64]{PreNBK}) showed that the closed sets in this topology (hence also the open sets) are in natural bijection with the hereditary torsion theories on $ {\rm Fun}\mbox{-}R$, equivalently with the subclasses of $ {\rm Mod}\mbox{-}R $ which are closed under products, pure submodules and pure-injective hulls.

\vspace{4pt}

\noindent 2. For the basis of a topology we may instead take those $ M\in {\mathcal A}(R)$, the smallest (not necessarily full) abelian subcategory of ${\rm Mod}\mbox{-}R$ generated by ${\rm mod}\mbox{-}R$ (\cite[\S6]{PreRajShv}).  This is the restriction of the Ziegler topology on ${\rm pinj}_R$ to ${\rm inj}_R$.

\begin{lemma}\label{injann3}\marginpar{injann3} If $ \phi /\psi  $ is a pp-pair in any number of free variables then $ \big(D\psi (_RR)/D\phi (_RR)\big)\,\cap\, {\rm inj}_R =(\phi /\psi )\cap {\rm inj}_R$.
\end{lemma}
\begin{proof} If $ E\in (\phi /\psi ) $ then there is $ \overline{a}\in E^n $ such that $ \overline{a} \in\phi (E)\setminus \psi (E)$, that is, such that $ \overline{a}.D\phi (R)=0 $ and $ \overline{a}.D\psi (R)\neq 0 $ (the dot here signifies scalar dot product of $ n$-tuples). Define $ D\psi (R)\rightarrow E $ by sending $ \overline{r}=(r_1,\dots,r_n) $ to $ \overline{a}.\overline{r}=\sum _1^na_ir_i $ - then this induces a well-defined non-zero morphism $ D\psi (R)/D\phi (R)\rightarrow E $ and so $ E\in (D\psi /D\phi )$.

For the converse, if $ f:D\psi (R)/D\phi (R)\rightarrow E $ is non-zero then there is an extension of it, $ f':\,_R R^n/D\phi (R)\rightarrow E$, say. Set $ \overline{a}=f'\overline{1}=(f'e_1,\dots, f'e_n) $ where $ e_1,\dots, e_n $ form a free basis of $ R^n$. Then for $ \overline{r}=(r_1,\dots, r_n)\in R^n $ we have $ f'\overline{r}=\overline{a}.\overline{r} $ and so $ \overline{a}.D\phi (R)=0 $ and $ \overline{a}.D\psi (R)\neq 0$, hence $ E\models \phi (\overline{a})\wedge \neg\psi (\overline{a})$; thus $ \overline{a} \in\phi (E)\setminus \psi (E)$ as required.
\end{proof}

If $ M\in {\mathcal A}(R) $ then, by definition and elementary duality, $ M=D\psi (R)/D\phi (R) $ for some pp-pair $ \phi /\psi  $ for right modules.  We saw above that $ (\phi /\psi )\cap {\rm inj}_R $ is a union of open sets of the form $ (R/(D\phi (R):r)) $ (intersected with $ {\rm inj}_R$), so $ (M)\cap {\rm inj}_R $ is a union of sets of the form $ (R/I)\cap {\rm inj}_R $ with $ I $ a pp-definable (from the left) right ideal, as claimed above.

If $ R $ is right coherent then the Ziegler-closed subsets of $ {\rm inj}_R $ are those of the form $ {\mathcal F}_\tau \cap {\rm inj}_R $ for $ \tau  $ a finite type torsion theory on $ {\rm Mod}\mbox{-}R $ (see \cite[5.1.11]{PreNBK}. Puninski's example, \cite[\S4]{GarkPre5}, shows that for a non-coherent (even commutative) ring not every Ziegler-closed subset of the Ziegler-closure of $ {\rm inj}_R $ in $ {\rm pinj}_R $ need have that form.  (If, however, we go up one representation level then, as described at the beginning of this section, for any ring $ R $ the Ziegler-closed subsets of $ {\rm pinj}_R $ are, {\it via} the embedding $ M\mapsto M\otimes_R(-)\in (R\mbox{-}{\rm mod},{\bf Ab}) $ into the functor category, exactly those of the form $ {\rm pinj}_R\cap {\mathcal F}_\tau  $ where $ \tau  $ is a finite type torsion theory).

\vspace{4pt}

\noindent 3. Successively coarser topologies on $ {\rm inj}_R $ of Ziegler type have (sub)bases:  $(M) $ with $ M $ finitely presented; $(R/I) $ with $ I $ finitely generated;
$(M) $ with $ M $ coherent, equivalently $ (R/I) $ with $ R/I $ coherent since if $ M $ is coherent then it will have a composition chain with cyclic coherent factors, $ R/I_1, \dots, R/I_n $ say, and then $ (M)\cap {\rm inj}_R =\bigcup _1^n (R/I_i)\cap {\rm inj}_R$; $(M) $ with $ M $ ${\rm FP}_\infty $ (indeed, there is a hierarchy obtained by restricting $ M $ to be $ {\rm FP}_n$ and varying $n$).

If $ R $ is right coherent then all these coincide with the Ziegler topology; Puninski's example referred to above shows that the second can be strictly coarser than the Ziegler topology. The second, that is, with basis the $ (R/I) $ where $ I $ is a finitely generated right ideal, is the {\bf Thomason} or {\bf Zariski}$^\ast $ {\bf topology} where $ ^\ast  $ refers to duality in the sense of \cite{Hoc}. This topology is dual to the fg-ideals Zariski-type topology, which has basis of opens the $ [R/I] $ with $ I $ finitely generated.  By \cite[2.2]{GarkPre5} the closed sets in the Zariski$^\ast $ topology are in natural bijection with the torsion theories of finite type on $ {\rm Mod}\mbox{-}R$, at least if $ R $ is commutative.  For more on these and for results which explain the terminology, see \cite{GarkPre3}, \cite{GarkPre5}.

\subsection{Zariski-type topologies on $ {\rm inj}_R$}

\vspace{4pt}

\noindent 1. The finest Zariski-type topology has, for a basis of opens, the $ [M]=\{ E\in {\rm inj}_R: (M,E)=0\}$ where $M$ can be any module. Given a module $ M$, the open set $ [M] $ is $ {\mathcal F}_\tau \cap {\rm inj}_R $ where $ \tau  $ is the hereditary torsion theory with torsion class generated by $ M$, and so arbitrary open sets are unions of such sets.

\vspace{4pt}

\noindent 2. Coarser is the topology with basis the $ [M] $ where $ M\in {\mathcal A}(R)$. If $ M=D\psi (R)/D\phi (R) $ then, by what has been seen already, $ [M]\cap {\rm inj}_R=[\phi /\psi ]\cap {\rm inj}_R $ so this is the rep-Zariski = dual-Ziegler topology.

We have seen that $ (M)\cap {\rm inj}_R $ is a union of sets of the form $ (R/I)\cap {\rm inj}_R$. If $ R $ is right coherent then $ (M)\cap {\rm inj}_R $ will be compact, so a finite union will do and, in that case, the $ [R/I] $ with $ R/I\in {\mathcal A}(R)$, equivalently with $ I\in {\mathcal A}(R)$, will give a basis.

\vspace{4pt}

\noindent 3. We may take just the $ [R/I] $ with $ I $ finitely generated for a basis for a topology. Then $ [R/I]\cap {\rm inj}_R $ will be the set of indecomposable injectives which are torsionfree for the (finite type) torsion theory on $ {\rm Mod}\mbox{-}R $ with torsion class generated by $ R/I$. Intersections of such sets, rather than their unions, seem to be more natural and it is the dual of this {\bf fg ideals topology} (${\rm inj}_{\rm fg}R $ in the notation of \cite{GarkPre5}) which seems to be the more interesting.

(In fact, that term, introduced in that paper, was applied just over commutative rings and had a somewhat different definition; we just check here that the definitions, there and here, are equivalent.  In \cite[\S3]{GarkPre5} we set $ D^{\rm fg}(I)=\{ E\in {\rm inj}_R: E \mbox{ is } {\rm gen}(R/I)\mbox{-torsionfree}\} $ where $ {\rm gen}(R/I) $ is the torsion theory with Gabriel filter generated by the powers $ I^n $ of $ I$. So it will be sufficient to show that $ D^{\rm fg}(I)=[R/I]$, that is, that $ (R/I,E)=0 $ implies $ (R/I^n,E)=0 $ for all $ n$. Suppose then, that we have $ f:R/I^n\rightarrow E $ non-zero,; we may choose $ n, f $ so that $ n $ is as small as possible. Then there is $ a\in I^{n-1}/I^n $ with $ fa\neq 0$; since $ {\rm ann}_R(a)\geq I $ we have an epimorphism $ R/I\rightarrow aR $ whose composition with $ f $ is non-zero, as required.)

\subsection{Intersection with $ {\rm Spec}(R)$}

Suppose that $ R $ is commutative.  Then $ P\mapsto E_P=E(R/P) $ embeds $ {\rm Spec}(R) $ into $ {\rm inj}_R $ and we will identify the former with its image in the latter.  If $ R $ is noetherian then the map is a bijection.
Clearly, if $ I $ is an ideal of the commutative ring $ R $ then $ (R/I)\cap \,{\rm Spec}(R)=\{ P\in {\rm Spec}(R): I\leq P\} = V(I) $ so $ [R/I]\cap \,{\rm Spec}(R) =D(I)$.  From the discussion and results quoted above we have the following.

\begin{cor}\label{intersectspec2}\marginpar{intersectspec2} For a commutative ring $ R $ the restrictions to $ {\rm Spec}(R) $ of the various Ziegler-type toplogies have, for their open sets, respectively the following types of set:

1. full support topology - $ \bigcup _\lambda  V(I_\lambda ) $ for any set $ \{ I_\lambda \}_\lambda  $ of ideals $ = $ the specialisation-closed sets;

2. Ziegler topology - $ \bigcup _\lambda  V(I_\lambda ) $ for any set $ \{ I_\lambda \}_\lambda  $ of pp-definable ideals, i.e.~with $ I_\lambda \in {\mathcal A}(R) $ for all $ \lambda $;

3. Zariski$^\ast $=Thomason topology - $ \bigcup _\lambda  V(I_\lambda ) $ for any set $ \{ I_\lambda \}_\lambda  $ of finitely generated ideals.

If $ R $ is coherent then 2.~and 3.~coincide; if $ R $ is noetherian then all three coincide.
\end{cor}

(Somewhat different notation was used in \cite[\S3.3,3.4]{GarkPre5}.)

We recall (\cite{PreZar}, see \cite[3.3, 3.5]{GarkPre5}) that if $ R $ is commutative and $ E\in {\rm inj}_R$ we set $ P(E)=\{ r\in R: ar=0 \mbox{ for some non-zero } a\in E\}$ - a prime ideal which is such that $ E $ and $ E_{P(E)} $ are topologically indistinguishable in both the Zariski$^\ast $=Thomason and the fg-ideals topologies.
Hence (\cite[3.6,3.7]{GarkPre5}) the embedding of $ {\rm Spec}(R) $ into $ {\rm inj}_R $ is a topological equivalence in both the Zariski$^\ast $=Thomason and the fg-ideals topologies.

It is also the case, \cite[Thm.~p.~395]{GarkPre5}, that for every commutative ring $ R $ there is a natural bijection between the torsion theories of finite type on $ {\rm Mod}\mbox{-}R $ and the open subsets of $ {\rm Spec}(R) $ equipped with the Zariski$^\ast $=Thomason topology.  It is shown in \cite{PreAxtFlat} that $ {\rm fun}(\langle {\rm Inj}\mbox{-}R\rangle) \simeq {\rm mod}\mbox{-}R $.

\begin{cor}\label{spandspec}\marginpar{spandspec} If $ R $ is commutative coherent then we have a natural bijection $ {\rm Sp}({\rm mod}\mbox{-}R) \leftrightarrow {\rm Spec}(R) $ between the $ \cap $-irreducible Serre subcategories of $ {\rm mod}\mbox{-}R $ and the prime ideals of $ R$.
\end{cor}

\end{document}